\documentclass[a4paper, final]{amsart}
\usepackage{amsmath,amsthm,amssymb,latexsym,epic,bbm,comment,color,mathrsfs}
\usepackage{graphicx,enumerate,stmaryrd}
\usepackage[all,2cell]{xy}
\usepackage{todonotes}
\xyoption{2cell}

\usepackage[active]{srcltx}
\usepackage[parfill]{parskip}
\UseTwocells

\newtheorem{theorem}{Theorem}
\newtheorem{lemma}[theorem]{Lemma}
\newtheorem{remark}[theorem]{Remark}
\newtheorem{corollary}[theorem]{Corollary}
\newtheorem{proposition}[theorem]{Proposition}
\newtheorem{example}[theorem]{Example}

\newcommand{\tto}{\twoheadrightarrow}

\font\sc=rsfs10

\newcommand{\cG}{\sc\mbox{G}\hspace{1.0pt}}

\begin{document}

\title{$G(\ell,k,d)$-modules via groupoids}
\author{Volodymyr Mazorchuk and Catharina Stroppel}

\begin{abstract}
In this note we describe a seemingly new approach to the complex representation theory 
of the wreath product $G\wr S_d$ where $G$ is a finite abelian group. The approach is 
motivated by an appropriate version of Schur-Weyl duality. We construct a 
combinatorially defined groupoid in which all endomorphism algebras are direct 
products of symmetric groups and prove that the groupoid algebra is isomorphic 
to the group algebra of $G\wr S_d$. This directly implies a classification of simple 
modules. As an application, we get a Gelfand model for $G\wr S_d$ from the classical 
involutive Gelfand model for the symmetric group. We describe the Schur-Weyl duality 
which motivates our approach and relate it to various Schur-Weyl dualities in the 
literature. Finally, we discuss an extension of these methods to all complex reflection 
groups of type $G(\ell,k,d)$.
\end{abstract}
\maketitle

\section{Introduction}\label{s1}
A very important class of finite groups are wreath products of the form $G\wr S_d$, where $S_d$ is the symmetric group
and $G$ is abelian. The study of its representation theory is a classical topic. The first major results, e.g. the classification of simple modules, were 
already obtained by Specht in his thesis \cite{Sp}. Since then the theory was revised on 
various occasions, in particular in case of $G$ being a cyclic group, see \cite{Ca}, \cite{Ke}, 
\cite{Os}, \cite{SUI} and references therein. This note contributes yet 
another approach which, from our point of view, simplifies the theory and makes 
several results, in particular, on the combinatorics of simple modules and on Gelfand 
models,  especially transparent.

Our approach originates in an attempt to understand various Schur-Weyl dualities 
appearing in  \cite{Wang}, \cite{ES1}, \cite{SaSt} in which on one side we have 
an action of a direct product of general linear groups while on the
other side we have a non-faithful action of the Coxeter group of type $B$, 
respectively $D$. While looking for similar results in the literature, we discovered
that analogous Schur-Weyl dualities already appeared in \cite{Re} and also
in the context of Ariki-Koike 
algebras in \cite{ATY}, \cite{Hu}, \cite{SS}, \cite{Sh}. 
These dualities have a common structure which suggest the substitution of the Coxeter group of type $B$ 
by a certain combinatorially defined groupoid, see Subsection~\ref{s2.2}
for a precise definition of the latter. The main observation of the present note is 
that this groupoid can be used to describe the representation theory of the Coxeter 
group of type $B$ and, more generally, of the wreath products of the form 
$\mathbf{C}_\ell\wr S_d$, where $\mathbf{C}_\ell$ is a cyclic group of order $\ell$ 
or any complex reflection group of type $G(\ell,k,d)$.
The transparent combinatorial structure of the groupoid proposes a straightforward 
reduction of all statement to type $A$, that is to the case of direct products of symmetric groups. 
An explicit construction of all irreducible representations for $G(\ell,k,d)$ can be found in 
Proposition~\ref{prop9} and Theorem~\ref{thm77}. The symmetric group $S_d$, the Weyl group of type $B_d$ 
and the Weyl group of type $D_d$ are the special examples $G(1,1,d)$,  $G(2,1,d)$,  and $G(2,2,d)$, respectively. 

In Subsection~\ref{s2.2} we define our main object of study, that is a finite groupoid 
$\cG_{(\ell,d)}$, and in Subsection~\ref{s2.6} we show that its algebra (over $\mathbb{C}$) 
is isomorphic to the group algebra of $\mathbf{C}_\ell\wr S_d$. Consequently, we immediately 
get a classification and explicit construction of simple $\mathbf{C}_\ell\wr S_d$-modules
(see Subsection~\ref{s2.4}) which does not even involve any counting of the number of 
conjugacy classes (the latter being one of the ingredients in all classical approaches). 
The indexing set of simple modules is the set of $\ell$-multi-partitions of $d$. Moreover, our 
construction immediately gives a basis of all simple modules indexed by all 
standard $\ell$-multi-tableaux of the corresponding type, see Subsection~\ref{s2.4}. 
We connect our construction of simple
modules to the one from \cite{SUI} which uses induction from generalized Young subgroups. 
Finally, we also provide in Subsection~\ref{s2.8} a straightforward construction of an involutive Gelfand model 
for $\mathbf{C}_\ell\wr S_d$ (that is a multiplicity free direct sum of all simple modules), 
significantly simplifying the previous approaches from 
\cite{APR2}, \cite{CF}. In Subsection~\ref{s3.4} we give a short proof of the Schur-Weyl duality 
which motivated our approach (as we mentioned before, several (quantum) versions of this 
duality exist in the literature). In Subsection~\ref{s3.8} we use this duality to justify 
that our results naturally extent to the case $G\wr S_d$, where $G$ is any finite abelian group.
Finally, in Section~\ref{s4} we extend most of the results to all complex reflection
groups $G(\ell,k,d)$. 

We note that our approach generalizes to the quantum group setting. 
However, to prevent that the main idea of the proof is buried 
in technical details, we stick to the non-quantized situation.
\vspace{2mm}

\noindent
{\bf Acknowledgements.} An essential part of the research was done during the visit of 
both authors to the Max Planck Institute for Mathematics in Bonn. We gratefully acknowledge 
hospitality and support by the MPIM. For the first author the research was partially supported 
by the Swedish Research Council, Knut and Alice Wallenbergs Stiftelse  and the 
Royal Swedish Academy of Sciences. We thank Daniel Tubbenhauer and Stuart Margolis for comments.

\section{Modules over generalized symmetric groups}\label{s2}

\subsection{Generalized symmetric groups}\label{s2.1}

We denote by $\mathbb{Z}\supseteq\mathbb{Z}_{\geq 0}\supseteq\mathbb{Z}_{>0}$ the sets of all integers, 
all nonnegative integers and all positive integers, respectively. For $n\in \mathbb{Z}_{\geq 0}$ we 
denote by $\underline{n}$ the set $\{1,2,\dots,n\}$ (with  $\underline{0}=\varnothing$). 
Throughout the paper we fix as ground field the field $\mathbb{C}$ of complex numbers and abbreviate  
$\otimes_{\mathbb{C}}$ as $\otimes$.

For $\ell\in \mathbb{Z}_{>0}$, let $\mathbf{C}_\ell$ be the group of all complex $\ell$-th roots 
of unity. The group $\mathbf{C}_\ell$ is cyclic and we fix some generator $\xi_\ell\in\mathbf{C}_\ell$, 
that is, a primitive $\ell$-th root of unity.

Given a set $X$, we denote by $S(X)$ the symmetric group on $X$ and abbreviate $S_d:=S(\underline{d})$ for any $d\in  \mathbb{Z}_{>0}$. For $\mathbf{d}=(d_1,d_2,\dots,d_k)\in \mathbb{Z}_{\geq 0}^k$ set
\begin{eqnarray}
\label{sd}
S_{\mathbf{d}}:=S_{d_1}\times S_{d_2}\times \dots\times S_{d_k}. 
\end{eqnarray}
Given $n\in\mathbb{Z}_{\geq 0}$ and a partition $\mu\vdash n$, we denote by $\mathscr{S}_{\mu}$ 
the (irreducible) Specht $S_n$-module corresponding to $\mu$. For 
$\mathbf{n}=(n_1,n_2,\dots,n_k)\in \mathbb{Z}_{\geq 0}^k$ and a multi-partition
$\boldsymbol{\mu}=(\mu_1,\mu_2,\dots,\mu_k)$ such that $\mu_i\vdash n_i$ for all $i$, 
we denote by $\mathscr{S}_{\boldsymbol{\mu}}$ the $S_{\mathbf{n}}$-module 
$\mathscr{S}_{\mu_1}\otimes\mathscr{S}_{\mu_2}\otimes\dots
\otimes\mathscr{S}_{\mu_k}$.

From now on we {\bf fix} $d\in\mathbb{Z}_{\geq 0}$  and $\ell\in \mathbb{Z}_{>0}$ and consider the wreath product $S(\ell,d):=\mathbf{C}_\ell\wr S_d$, 
also known as a {\em generalized symmetric group}. The group $S(\ell,d)$ is naturally 
identified with the group of all complex $d\times d$-matrices $X$ which satisfy the following two conditions:
\begin{itemize}
\item Each row and each column of $X$ contains exactly one non-zero entry.
\item Each non-zero entry of $X$ is an element of $\mathbf{C}_\ell$.
\end{itemize}
We have $|S(\ell,d)|=\ell^d\cdot d!$. The group $S(\ell,d)$
is a complex reflection group usually denoted by $G(\ell,1,n)$.

The group $S(\ell,d)$ has a presentation with generators $s_0,s_1,\dots,s_{d-1}$ and relations
\begin{equation}\label{eq3}
\begin{array}{rclr}
s_0^{\ell}&=&e;\\
s_i^2&=&e,&  i=1,2,\dots,d-1;\\
s_0s_1s_0s_1&=&s_1s_0s_1s_0;\\ 
s_is_{i+1}s_i&=&s_{i+1}s_i s_{i+1},& i=1,2,\dots,d-2;\\
s_is_j&=&s_js_i,& |i-j|>1,\,\,\, i,j=0,1,\dots,d-1.
\end{array}
\end{equation}
An isomorphism with the earlier description
is given by sending $s_0$ to the diagonal $d\times d$-matrix in 
which the $(1,1)$-entry is $\xi_{\ell}$ and all other diagonal entries are equal to $1$, and 
sending $s_i$ for $i=1,2,\dots,d-1$ to the permutation matrix corresponding to the transposition $(i,i+1)$.

Some classical special cases: The group $S(1,d)\cong S_d$ is the Weyl group of type $A_{d-1}$ 
and $S(2,d)$ is the Weyl group of type $B_d$ and $C_d$. 

\subsection{The groupoid $\cG_{(\ell,d)}$}\label{s2.2}

Consider a category $\cG_{(\ell,d)}$ defined as follows:
\begin{itemize}
\item Objects of $\cG_{(\ell,d)}$ are all maps $f:\underline{d}\to\underline{\ell}$.
\item For two objects $f$ and $g$ the set of morphisms $\cG_{(\ell,d)}(f,g)$ consists of all bijections
$\sigma:\underline{d}\to\underline{d}$ such that $g\circ \sigma=f$.
\item The identity morphism $e_f\in\cG_{(\ell,d)}(f,f)$ is the identity map 
$\mathrm{Id}_{\underline{d}}:\underline{d}\to\underline{d}$.
\item Composition of morphisms is given by composition of maps.
\end{itemize}
It is convenient to think of objects in $\cG_{(\ell,d)}$ as $\underline{\ell}$-colorings of elements 
in $\underline{d}$, that is, as ordered sequences of $d$ dots colored in $\ell$ colors. Then 
morphisms in $\cG_{(\ell,d)}$ are color preserving bijections. We usually 
represent them in terms of colored permutation diagrams (read from top to bottom), see Figure~\ref{fig2}.

\begin{example}\label{ex1}
{\rm
Let $d=2$ and $\ell=2$. We depict colors as follows: $1=\color{red}{\text{red}}$ and 
$2=\color{blue}{\text{blue}}$. Then $\cG_{(\ell,d)}$ has four objects, namely
\begin{displaymath}
\begin{picture}(30,5)
\put(05,02.50){\makebox(0,0)[cc]{\color{red}{$\circ$}}}
\put(25,02.50){\makebox(0,0)[cc]{\color{red}{$\circ$}}}
\end{picture}\qquad\qquad\qquad
\begin{picture}(30,5)
\put(05,02.50){\makebox(0,0)[cc]{\color{red}{$\circ$}}}
\put(25,02.50){\makebox(0,0)[cc]{\color{blue}{$\bullet$}}}
\end{picture}\qquad\qquad\qquad
\begin{picture}(30,5)
\put(05,02.50){\makebox(0,0)[cc]{\color{blue}{$\bullet$}}}
\put(25,02.50){\makebox(0,0)[cc]{\color{red}{$\circ$}}}
\end{picture}\qquad\qquad\qquad
\begin{picture}(30,5)
\put(05,02.50){\makebox(0,0)[cc]{\color{blue}{$\bullet$}}}
\put(25,02.50){\makebox(0,0)[cc]{\color{blue}{$\bullet$}}}
\end{picture}
\end{displaymath}
and the elements of $\cG_{(\ell,d)}(f,g)$ are given in Figure~\ref{fig2}
(for convenience, all red strands are dashed and all red points are circled).
\begin{figure}
\begin{displaymath}
\begin{array}{c||c|c|c|c}
g\backslash f&\begin{picture}(30,5)
\put(05,02.50){\makebox(0,0)[cc]{\color{red}{$\circ$}}}
\put(25,02.50){\makebox(0,0)[cc]{\color{red}{$\circ$}}}
\end{picture}&
\begin{picture}(30,5)
\put(05,02.50){\makebox(0,0)[cc]{\color{red}{$\circ$}}}
\put(25,02.50){\makebox(0,0)[cc]{\color{blue}{$\bullet$}}}
\end{picture}&
\begin{picture}(30,5)
\put(05,02.50){\makebox(0,0)[cc]{\color{blue}{$\bullet$}}}
\put(25,02.50){\makebox(0,0)[cc]{\color{red}{$\circ$}}}
\end{picture}&
\begin{picture}(30,5)
\put(05,02.50){\makebox(0,0)[cc]{\color{blue}{$\bullet$}}}
\put(25,02.50){\makebox(0,0)[cc]{\color{blue}{$\bullet$}}}
\end{picture}\\
\hline\hline
\begin{picture}(30,30)
\put(05,02.50){\makebox(0,0)[cc]{\color{red}{$\circ$}}}
\put(25,02.50){\makebox(0,0)[cc]{\color{red}{$\circ$}}}
\end{picture}&
\begin{picture}(80,20)
\put(05,02.50){\makebox(0,0)[cc]{\color{red}{$\circ$}}}
\put(25,02.50){\makebox(0,0)[cc]{\color{red}{$\circ$}}}
\put(05,17.50){\makebox(0,0)[cc]{\color{red}{$\circ$}}}
\put(25,17.50){\makebox(0,0)[cc]{\color{red}{$\circ$}}}
\put(55,02.50){\makebox(0,0)[cc]{\color{red}{$\circ$}}}
\put(75,02.50){\makebox(0,0)[cc]{\color{red}{$\circ$}}}
\put(55,17.50){\makebox(0,0)[cc]{\color{red}{$\circ$}}}
\put(75,17.50){\makebox(0,0)[cc]{\color{red}{$\circ$}}}
\color{red}{\dottedline{1}(05.00,02.50)(05.00,17.50)
\dottedline{1}(25.00,02.50)(25.00,17.50)
\dottedline{1}(55.00,02.50)(75.00,17.50)
\dottedline{1}(75.00,02.50)(55.00,17.50)}
\end{picture}
&&&\\
\hline
\begin{picture}(30,30)
\put(05,02.50){\makebox(0,0)[cc]{\color{red}{$\circ$}}}
\put(25,02.50){\makebox(0,0)[cc]{\color{blue}{$\bullet$}}}
\end{picture}&&
\begin{picture}(30,20)
\color{red}{
\put(05,02.50){\makebox(0,0)[cc]{$\circ$}}
\put(05,17.50){\makebox(0,0)[cc]{$\circ$}}
\dottedline{1}(05.00,02.50)(05.00,17.50)}
\color{blue}{
\put(25,02.50){\makebox(0,0)[cc]{$\bullet$}}
\put(25,17.50){\makebox(0,0)[cc]{$\bullet$}}
\drawline(25.00,02.50)(25.00,17.50)}
\end{picture}
&
\begin{picture}(30,20)
\color{blue}{
\put(05,17.50){\makebox(0,0)[cc]{$\bullet$}}
\put(25,02.50){\makebox(0,0)[cc]{$\bullet$}}
\drawline(25.00,02.50)(05.00,17.50)}
\color{red}{
\put(02,02.50){\makebox(0,0)[cc]{$\circ$}}
\put(22,17.50){\makebox(0,0)[cc]{$\circ$}}
\dottedline{1}(02,02.50)(22,17.50)}
\end{picture}
&\\
\hline
\begin{picture}(30,30)
\put(05,02.50){\makebox(0,0)[cc]{\color{blue}{$\bullet$}}}
\put(25,02.50){\makebox(0,0)[cc]{\color{red}{$\circ$}}}
\end{picture}&&
\begin{picture}(30,20)
\color{red}{
\put(05,17.50){\makebox(0,0)[cc]{$\circ$}}
\put(25,02.50){\makebox(0,0)[cc]{$\circ$}}
\dottedline{1}(25.00,02.50)(05.00,17.50)}
\color{blue}{
\put(02,02.50){\makebox(0,0)[cc]{$\bullet$}}
\put(22,17.50){\makebox(0,0)[cc]{$\bullet$}}
\drawline(02,02.50)(22,17.50)}
\end{picture}
&
\begin{picture}(30,20)
\color{blue}{
\put(05,02.50){\makebox(0,0)[cc]{$\bullet$}}
\put(05,17.50){\makebox(0,0)[cc]{$\bullet$}}
\drawline(05.00,02.50)(05.00,17.50)}
\color{red}{
\put(25,02.50){\makebox(0,0)[cc]{$\circ$}}
\put(25,17.50){\makebox(0,0)[cc]{$\circ$}}
\dottedline{1}(25.00,02.50)(25.00,17.50)}
\end{picture}
&\\
\hline
\begin{picture}(30,30)
\put(05,02.50){\makebox(0,0)[cc]{\color{blue}{$\bullet$}}}
\put(25,02.50){\makebox(0,0)[cc]{\color{blue}{$\bullet$}}}
\end{picture}
&&&&
\begin{picture}(80,20)
\put(05,02.50){\makebox(0,0)[cc]{\color{blue}{$\bullet$}}}
\put(25,02.50){\makebox(0,0)[cc]{\color{blue}{$\bullet$}}}
\put(05,17.50){\makebox(0,0)[cc]{\color{blue}{$\bullet$}}}
\put(25,17.50){\makebox(0,0)[cc]{\color{blue}{$\bullet$}}}
\put(55,02.50){\makebox(0,0)[cc]{\color{blue}{$\bullet$}}}
\put(75,02.50){\makebox(0,0)[cc]{\color{blue}{$\bullet$}}}
\put(55,17.50){\makebox(0,0)[cc]{\color{blue}{$\bullet$}}}
\put(75,17.50){\makebox(0,0)[cc]{\color{blue}{$\bullet$}}}
\color{blue}{\drawline(05.00,02.50)(05.00,17.50)
\drawline(25.00,02.50)(25.00,17.50)
\drawline(55.00,02.50)(75.00,17.50)
\drawline(75.00,02.50)(55.00,17.50)}
\end{picture}
\\
\end{array}
\end{displaymath}
\caption{The sets of morphisms $\cG_{(\ell,d)}(f,g)$ from Example~\ref{ex1}.}\label{fig2}
\end{figure}
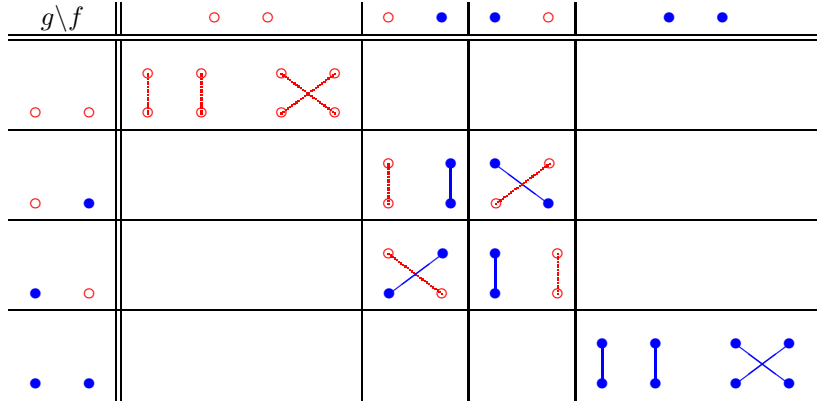
}
\end{example}

For an object $f\in \cG_{(\ell,d)}$ the {\em type} of $f$ is defined as 
\begin{displaymath}
\boldsymbol{\lambda}_f:=(\lambda_1,\lambda_2,\dots,\lambda_\ell) 
\end{displaymath}
where $\lambda_i:=|\{x\in\underline{d}\mid  f(x)=i\}|$ for $i=1,2,\dots,\ell$. 
As $\lambda_1+\lambda_2+\dots+\lambda_\ell=d$ and all $\lambda_i\in\mathbb{Z}_{\geq 0}$,
the type $\boldsymbol{\lambda}_f$ is a {\em composition of $d$ with $\ell$ parts}. We denote by
$\Lambda(\ell,d)$ the set of all compositions of $d$ with $\ell$ parts, that is all 
$(a_1,a_2,\dots,a_\ell)\in \mathbb{Z}_{\geq 0}^{\ell}$ such that $a_1+a_2+\dots+a_\ell=d$.
Then $\Lambda(\ell,d)$ is exactly the set of all possible types for
objects in $\cG_{(\ell,d)}$.

\begin{proposition}\label{prop3}
Let $f,g\in \cG_{(\ell,d)}$.

\begin{enumerate}[$($i$)$]
\item\label{prop3.1} We have $\cG_{(\ell,d)}(f,g)\neq \varnothing$ if and only if $\boldsymbol{\lambda}_f=\boldsymbol{\lambda}_g$.
\item\label{prop3.2} If $\boldsymbol{\lambda}_f=\boldsymbol{\lambda}_g=
(\lambda_1,\lambda_2,\dots,\lambda_\ell)=:\boldsymbol{\lambda}$, 
then $\displaystyle |\cG_{(\ell,d)}(f,g)|=\prod_{i=1}^\ell\lambda_i!=:\boldsymbol{\lambda}!$.
\item\label{prop3.3}  We have $\displaystyle \left|\coprod_{f_1,f_2\in\cG_{(\ell,d)}}\cG_{(\ell,d)}(f_1,f_2)\right|
=\ell^d\cdot d!$, where ${f_1,f_2}\in  \cG_{(\ell,d)}$.
\end{enumerate}
\end{proposition}
\begin{proof}
Claims~\eqref{prop3.1} and \eqref{prop3.2} follow directly from the definitions.
Claim~\eqref{prop3.3} follows from the observation that there is a natural bijection 
between the set of all morphisms in $\cG_{(\ell,d)}$ and the set of all colored permutations of $d$ dots with $\ell$ colors. 
\end{proof}

Note that $\cG_{(\ell,d)}$ is a groupoid. For $\boldsymbol{\lambda}\in \Lambda(\ell,d)$ denote by
$\cG_{(\ell,d)}^{\boldsymbol{\lambda}}$ the full subcategory of $\cG_{(\ell,d)}$ generated by all 
objects of type $\boldsymbol{\lambda}$. Then $\cG_{(\ell,d)}^{\boldsymbol{\lambda}}$ is a 
subgroupoid and is a connected component in $\cG_{(\ell,d)}$ in the sense that 
$\cG_{(\ell,d)}^{\boldsymbol{\lambda}}(f,g)\neq \varnothing$ for any
$f,g\in \cG_{(\ell,d)}^{\boldsymbol{\lambda}}$ and
\begin{displaymath}
\cG_{(\ell,d)}=\coprod_{\boldsymbol{\lambda}\in\Lambda(\ell,d)}\cG_{(\ell,d)}^{\boldsymbol{\lambda}}. 
\end{displaymath}
If $f$ has type $\boldsymbol{\lambda}$, then there is an isomorphism of groups $\cG_{(\ell,d)}(f,f)\cong S_{\boldsymbol{\lambda}}$.

\subsection{The linearization of $\cG_{(\ell,d)}$}\label{s2.3}

For a set $X$ we denote by $\mathbb{C}[X]$ the formal complex vector space with the elements of $X$ 
as  basis. If $X=\varnothing$, then we have $\mathbb{C}[X]=0$.

Note that the groupoid $\cG_{(\ell,d)}$ is a finite category. Denote by $\mathbb{C}\cG_{(\ell,d)}$ the 
$\mathbb{C}$-linear category generated by $\cG_{(\ell,d)}$. This means the following:
\begin{itemize}
\item Objects in $\mathbb{C}\cG_{(\ell,d)}$ are the same as in $\cG_{(\ell,d)}$.
\item For two objects $f$ and $g$ we have $\mathbb{C}\cG_{(\ell,d)}(f,g):=
\mathbb{C}[\cG_{(\ell,d)}(f,g)]$.
\item The identity elements in $\mathbb{C}\cG_{(\ell,d)}$ are given by the identity elements 
in $\cG_{(\ell,d)}$.
\item Composition in $\mathbb{C}\cG_{(\ell,d)}$ is induced from composition in $\cG_{(\ell,d)}$ 
by bilinearity.
\end{itemize}
We remark that $\mathbb{C}\cG_{(\ell,d)}$ is no longer a groupoid. 
If $f\in \mathbb{C}\cG_{(\ell,d)}$ is of type $\boldsymbol{\lambda}$, 
then $\mathbb{C}\cG_{(\ell,d)}(f,f)\cong \mathbb{C}[S_{\boldsymbol{\lambda}}]$, 
the group algebra of $S_{\boldsymbol{\lambda}}$.

\subsection{Simple finite dimensional $\mathbb{C}\cG_{(\ell,d)}$-modules}\label{s2.4}

Consider the category $\mathbb{C}\cG_{(\ell,d)}\text{-}\mathrm{mod}$ of $\mathbb{C}$-linear 
functors from $\mathbb{C}\cG_{(\ell,d)}$ to the category of finite dimensional complex 
vector spaces. Objects in $\mathbb{C}\cG_{(\ell,d)}\text{-}\mathrm{mod}$ are called 
{\em $\mathbb{C}\cG_{(\ell,d)}$-modules}. Morphism in  $\mathbb{C}\cG_{(\ell,d)}\text{-}\mathrm{mod}$
are natural transformations of functors.

For $\boldsymbol{\lambda}\in \Lambda(\ell,d)$ consider the set $\mathbf{T}_{\boldsymbol{\lambda}}$
of all $\ell$-multi-partitions $\mathbf{p}$ of shape $\boldsymbol{\lambda}$, that is all 
$\ell$-tuples $\mathbf{p}=(p_1,p_2,\dots,p_\ell)$  of partitions such that $p_i\vdash \lambda_i$
for all $i$. For $f,g$ of type $\boldsymbol{\lambda}$ let $\sigma_{(f,g)}$ be the unique 
{\em order preserving} element in $\cG_{(\ell,d)}(f,g)$ in the sense that it has the 
following property:
\begin{center}
For all $i,j\in \underline{d}$ satisfying $i<j$ and $f(i)=f(j)$, we have $\sigma_{(f,g)}(i)<\sigma_{(f,g)}(j)$.
\end{center}
The element $\sigma_{(f,g)}$ should be thought of as a ``canonical'' isomorphism between
the objects $f$ and $g$. Clearly, for $f,g,h$ of type $\boldsymbol{\lambda}$ we have $\sigma_{(g,h)}\sigma_{(f,g)}=\sigma_{(f,h)}$ and
$\sigma_{(f,f)}=e_f$. We also denote by $f_{\boldsymbol{\lambda}}$ the unique object of type $\boldsymbol{\lambda}$ in which all colors are assigned in the natural order from $1$ to $d$,
that is
\begin{gather*}
f_{\boldsymbol{\lambda}}(1)=f_{\boldsymbol{\lambda}}(2)=\dots=
f_{\boldsymbol{\lambda}}(\lambda_1)=1,\\
f_{\boldsymbol{\lambda}}(\lambda_1+1)=f_{\boldsymbol{\lambda}}(\lambda_1+2)=\dots=
f_{\boldsymbol{\lambda}}(\lambda_1+\lambda_2)=2,\\
\dots.
\end{gather*}
This element is our ``canonical'' object in $\cG_{(\ell,d)}^{\boldsymbol{\lambda}}$.
We fix the evident identification of $S_{\boldsymbol{\lambda}}$ with 
$\cG_{(\ell,d)}(f_{\boldsymbol{\lambda}},f_{\boldsymbol{\lambda}})$.

For $\mathbf{p}\in \mathbf{T}_{\boldsymbol{\lambda}}$, define a $\mathbb{C}\cG_{(\ell,d)}$-module 
$\mathrm{L}_{\mathbf{p}}$ as follows:
\begin{itemize}
\item $\mathrm{L}_{\mathbf{p}}(f):=
\begin{cases}
\mathscr{S}_{\mathbf{p}},&\text{if $f$ is of type $\boldsymbol{\lambda}$}; \\
0,&\text{if $f$ is not of type $\boldsymbol{\lambda}$}.
\end{cases}
$
\item For any $f,g$ of type $\lambda$, any $\pi\in\cG_{(\ell,d)}(f,g)$ and 
any $v\in \mathscr{S}_{\mathbf{p}}$, we set 
\begin{displaymath}
\mathrm{L}_{\mathbf{p}}(\pi)\cdot v:=\sigma_{g,f_{\boldsymbol{\lambda}}}\pi\sigma_{f_{\boldsymbol{\lambda}},f}(v).
\end{displaymath}
and extend this linearly to an action of the whole of $\mathbb{C}\cG_{(\ell,d)}$.
\end{itemize}

\begin{proposition}\label{prop9}
We have the following:
\begin{enumerate}[$($i$)$] 
\item\label{prop9.1}
For $\mathbf{p}\in \mathbf{T}_{\boldsymbol{\lambda}}$, the functor $\mathrm{L}_{\mathbf{p}}$ defined above is a simple $\mathbb{C}\cG_{(\ell,d)}$-module.
\item\label{prop9.2} The set 
\begin{displaymath}
\coprod_{\boldsymbol{\lambda}\in\Lambda(\ell,d)}\left\{\mathrm{L}_{\mathbf{p}}\mid 
\mathbf{p}\in \mathbf{T}_{\boldsymbol{\lambda}}\right\}
\end{displaymath}
is a cross-section of isomorphism classes of simple $\mathbb{C}\cG_{(\ell,d)}$-modules.
\end{enumerate}
\end{proposition}

\begin{proof}
The functoriality of $\mathrm{L}_{\mathbf{p}}$ follows directly from the definitions and 
the observation that $\sigma_{f_{\boldsymbol{\lambda}},g}\sigma_{g,f_{\boldsymbol{\lambda}}}=e_g$. 
The simplicity of $\mathrm{L}_{\mathbf{p}}$ follows by construction from the facts that 
$\cG_{(\ell,d)}$ is a groupoid and that $\mathscr{S}_{\boldsymbol{\lambda}}$ is a simple 
$\cG_{(\ell,d)}(f_{\boldsymbol{\lambda}},f_{\boldsymbol{\lambda}})$-module.
This proves claim~\eqref{prop9.1}.

Claim~\eqref{prop9.2} follows from the facts that connected components of the groupoid $\cG_{(\ell,d)}$ are indexed by $\boldsymbol{\lambda}\in\Lambda(\ell,d)$ and that the set
$\{\mathscr{S}_{\mathbf{p}}\mid \mathbf{p}\in \mathbf{T}_{\boldsymbol{\lambda}}\}$
is a complete and irredundant set of representatives of
isomorphism classes of simple $\mathbb{C}\cG_{(\ell,d)}(f_{\boldsymbol{\lambda}},f_{\boldsymbol{\lambda}})$-modules
(see e.g. \cite[Chapter~2]{Sa}). This completes the proof.
\end{proof}

The vector space
\begin{displaymath}
A_{(\ell,d)}:=\bigoplus_{f,g\in\cG_{(\ell,d)}} \mathbb{C}\cG_{(\ell,d)}(f,g)
\end{displaymath}
inherits from $\mathbb{C}\cG_{(\ell,d)}$ the structure of a finite dimensional associative 
algebra over $\mathbb{C}$. As usual, there is a canonical equivalence of categories 
\begin{equation}\label{eq1}
\mathbb{C}\cG_{(\ell,d)}\text{-}\mathrm{mod}\cong A_{(\ell,d)}\text{-}\mathrm{mod},
\end{equation}
where the right hand side denotes the category of finite dimensional $A_{(\ell,d)}$-modules.
Taking this equivalence into account, Proposition~\ref{prop9} provides an explicit description 
of  all simple $A_{(\ell,d)}$-modules. In what follows we identify $A_{(\ell,d)}$-modules and 
$\mathbb{C}\cG_{(\ell,d)}$-modules via this equivalence. We note that
\begin{equation}\label{eq5}
\sum_{f\in \cG_{(\ell,d)}}\dim(\mathrm{L}_{\mathbf{p}}(f))=\frac{n!}{\boldsymbol{\lambda}!}\dim \mathscr{S}_{\mathbf{p}}.
\end{equation}
For $\mu\vdash n$, the Specht $S_n$-module $\mathscr{S}_{\mu}$ has a basis given by 
polytabloids corresponding to standard Young tableaux of shape $\mu$, see \cite[Section~2.5]{Sa}.
This extends in the evident way to a basis in $\mathscr{S}_{\mathbf{p}}$ and hence gives a 
basis in each $\mathrm{L}_{\mathbf{p}}(f)$, where $f$ is of type $\boldsymbol{\lambda}$. A 
more ``natural'' parameterization of the elements of this basis in the space
$\mathrm{L}_{\mathbf{p}}(g)$ is obtained by applying $\sigma_{f_{\boldsymbol{\lambda}},g}$
to the entries of the corresponding standard tableaux.

\subsection{Connection to $S(\ell,d)$}\label{s2.6}

For $f,g\in \cG_{(\ell,d)}$ and $\sigma\in \cG_{(\ell,d)}(f,g)$ we will write $\sigma=\sigma_{(f,g)}$ 
to distinguish it from the same $\sigma$ appearing as a morphism between another pair of 
objects. If $\sigma\not\in \cG_{(\ell,d)}(f,g)$, we write $\sigma_{(f,g)}=0$ (viewing it
as an element in $\mathbb{C}\cG_{(\ell,d)}(f,g)$).

There is a unique linear map $\Phi:\mathbb{C}[S(\ell,d)]\to A_{(\ell,d)}$ such that
\begin{itemize}
\item $\Phi(\sigma)=\sum_{f,g}\sigma_{(f,g)}$ for $\sigma\in S_d$;
\item $\Phi(s_0)=\sum_{f}\xi_\ell^{f(1)}e_f$.
\end{itemize} 

\begin{theorem}\label{thm7}
The map $\Phi:\mathbb{C}[S(\ell,d)]\to A_{(\ell,d)}$ is an isomorphism of algebras. 
\end{theorem}

\begin{proof}
Note that we defined the map on algebra generators. To prove that $\Phi$ is a well-defined homomorphism, it is enough to check that 
$\Phi(s_i)$, where $i=0,1,2,\dots,d-1$, satisfy the defining relations for $S(\ell,d)$ 
given in \eqref{eq3}. All defining relations which do not involve $s_0$ are clear from 
the definition. That $\Phi(s_0)^\ell$ is the identity follows from the fact that $\xi_\ell^\ell=1$.
That $\Phi(s_0)\Phi(s_i)=\Phi(s_i)\Phi(s_0)$ for $i\neq 1$ is clear because the definition 
of $\Phi(s_0)$ only involves $f(1)$. 

It remains to verify that 
$\Phi(s_0)\Phi(s_1)\Phi(s_0)\Phi(s_1)=\Phi(s_1)\Phi(s_0)\Phi(s_1)\Phi(s_0)$.
It is straightforward to show that both sides of this equality are equal to the element
$\sum_{f}\xi_\ell^{f(1)+f(2)}e_f$. This implies that $\Phi$ is a homomorphism.

Note that $\dim(\mathbb{C}S(\ell,d))=\dim(A_{(\ell,d)})$. Therefore to complete the proof it is 
enough to check, say, surjectivity of $\Phi$. Since each $\sigma_{(f,g)}$ appears in 
$\Phi(\sigma)$ with a non-zero coefficient, it is enough to check that the identity morphism 
$e_f$ is in the image of $\Phi$ for each $f$.

Denote by $B$ the subalgebra of $A_{(\ell,d)}$ generated by all $e_f$, where $f\in \cG_{(\ell,d)}$.
The algebra $B$ is a commutative split semi-simple $\mathbb{C}$-algebra of dimension $l^d$.
We set $B':=\Phi(\mathbb{C}S(\ell,d))\cap B$. Then $B'$ is a unital subalgebra of $B$ and we 
need to show that $B'=B$. The left multiplication with the element $\Phi(s_0)\in B'$ on $B$ 
has different eigenvalues $\xi_\ell^1,\xi_\ell^2,\dots,\xi_\ell^{\ell}$. Therefore, by taking polynomials 
in $\Phi(s_0)$, we get that $B'$ contains, for each $s=1,2,\dots,\ell$, the element
\begin{displaymath}
x_s:=\sum_{f:f(1)=s}e_f.
\end{displaymath}

{\it Claim:}  Let now $s\in\underline{\ell}$ be fixed. Then $e_f\in B'$ for each $f$ with $f(1)=s$.

We prove the claim by downward induction on $m=|\{i\in\underline{d}\mid f(i)=s\}|$.
Assume first that $m=d$, that is $f=(s,s,\dots,s)$. Note that 
\begin{displaymath}
\Phi(\sigma^{-1}s_0\sigma)=
\Phi(\sigma^{-1})\Phi(s_0)\Phi(\sigma)\in B'\quad\text{ for any }\quad \sigma\in S_d.
\end{displaymath} 
As $x_s$ is a polynomial in $\Phi(s_0)$, we get $\Phi(\sigma^{-1})x_s\Phi(\sigma)\in B'$
for each $\sigma\in S_d$. Since $B'$ is a subalgebra of $B$, we have 
\begin{displaymath}
\prod_{\sigma\in S_d}\Phi(\sigma^{-1}) x_s\Phi(\sigma)=\xi_\ell^{a} e_{(s,s,\dots,s)}\in B'
\end{displaymath}
for some $a\in\mathbb{Z}_{>0}$. This implies $e_{(s,s,\dots,s)}\in B'$ and the
basis of the induction is established. 

Now we prove the induction step. Consider the set $X=\{i\in\underline{d}\mid f(i)=s\}$.
Then, similarly to the above, we have
\begin{displaymath}
\prod_{\sigma\in S(X)}\Phi(\sigma^{-1}) x_s\Phi(\sigma)=\sum_{g}\xi_\ell^{a_g}e_g\in B'
\end{displaymath}
for some $a_g\in\mathbb{Z}_{>0}$, where the sum on the right hand side is taken over all 
$g$ such that $g(i)=s$ for each $i\in X$.  If $g\neq f$, then 
\begin{displaymath}
|\{i\in\underline{d}\mid g(i)=s\}|>|\{i\in\underline{d}\mid f(i)=s\}|
\end{displaymath}
and hence, by induction, $e_g\in B'$. Therefore $e_f\in B'$ and the proof is complete.
\end{proof}

The isomorphism $\Phi$ from Theorem \ref{thm7} induces an equivalence of categories
\begin{equation}
\label{equiv}
\overline{\Phi}:A_{(\ell,d)}\text{-}\mathrm{mod}\to \mathbb{C}[S(\ell,d)]\text{-}\mathrm{mod}.
\end{equation}
 Combined  with Subsection~\ref{s2.4}, $\overline{\Phi}$ provides a very natural and neat description 
of simple $\mathbb{C}[S(\ell,d)]$-modules. For alternative descriptions of simple $\mathbb{C}[S(\ell,d)]$-modules
we refer the reader to \cite{Ca}, \cite{Ke}, \cite{Os}, \cite{SUI}, \cite{Sp} and references therein.

\subsection{Simple $\mathbb{C}[S(\ell,d)]$-modules via generalized Young subgroups}\label{s2.7}

Here we connect our approach with the one in \cite{SUI}. For $j=1,2,\dots,n$, denote by 
$s_0^{(j)}$ the element $s_{j-1}s_{j-2}\dots s_1s_0s_1s_2\dots s_{j-1}\in S(\ell,d)$.
This element is the diagonal $d\times d$ matrix in which the $(j,j)$-entry equals $\xi_\ell$
and all other diagonal entries are $1$. We have $s_0=s_0^{(1)}$. It is easy to check that 
\begin{equation}\label{eqn1}
\Phi(s_0^{(j)})=\sum_{f\in\cG_{(\ell,d)}}\xi_\ell^{f(j)}e_f.
\end{equation}

Let $\boldsymbol{\lambda}\in\Lambda(\ell,d)$ and 
$f\in \cG_{(\ell,d)}^{\boldsymbol{\lambda}}$. Then we have the decomposition
\begin{displaymath}
\underline{n}=X_1^f\cup X_2^f\cup\dots\cup X_\ell^f, 
\end{displaymath}
where $X^f_i:=\{j\in\underline{n}\mid  f(j)=i\}$ for $i=1,2,\dots,\ell$. Denote by $G_i^{f}$ 
the subgroup of $S(\ell,d)$ generated by all $s_0^{(j)}$, where $j\in X^f_i$, and also by 
all permutations of $X^f_i$ which fix all points outside $X^f_i$. The subgroup $G_i^{f}$ 
is isomorphic to $S(l,|X^f_i|)$ and the direct product 
\begin{displaymath}
G^{f}:=G_1^{f}\times G_2^{f}\times \dots \times G_\ell^{f} 
\end{displaymath}
is, naturally, a subgroup of $S(\ell,d)$. The subgroup $G^{f}$ is called a 
{\em generalized Young subgroup} of $S(l,n)$.

\begin{lemma}\label{lem31}
Let $\boldsymbol{\lambda}\in\Lambda(\ell,d)$, $\mathbf{p}\in\mathbf{T}_{\boldsymbol{\lambda}}$ 
and  $f\in \cG_{(\ell,d)}^{\boldsymbol{\lambda}}$.

\begin{enumerate}[$($i$)$]
\item\label{lem31.1} The space $\mathrm{L}_{\mathbf{p}}(f)$ inherits the structure of a 
simple $G^{f}$-module by restriction.
\item\label{lem31.2} The $S(l,n)$-modules $\overline{\Phi}(\mathrm{L}_{\mathbf{p}})$ and
$\mathrm{Ind}_{G^{f}}^{S(\ell,d)}\,\mathrm{L}_{\mathbf{p}}(f)$ are isomorphic.
\end{enumerate} 
\end{lemma}

\begin{proof}
That $\mathrm{L}_{\mathbf{p}}(f)$ is stable under the action of all $s_0^{(j)}$ is clear 
from the definitions. Similarly, it is also clear that $\mathrm{L}_{\mathbf{p}}(f)$ is 
stable under the action of all permutations which preserve colors. Claim~\eqref{lem31.1} follows.

From claim~\eqref{lem31.1}, it follows by adjunction that 
$\mathrm{Ind}_{G^{f}}^{S(\ell,d)}\,\mathrm{L}_{\mathbf{p}}(f)$ surjects onto 
$\overline{\Phi}(\mathrm{L}_{\mathbf{p}})$. However, since the index of $G^{f}$ in $S(\ell,d)$ 
equals  $\frac{n!}{\boldsymbol{\lambda}!}$, from \eqref{eq5} it follows that the modules 
$\overline{\Phi}(\mathrm{L}_{\mathbf{p}})$ and $\mathrm{Ind}_{G^{f}}^{S(\ell,d)}\,\mathrm{L}_{\mathbf{p}}(f)$
have the same dimension and thus are isomorphic. 
\end{proof}

Using this basis in each $\mathrm{L}_{\mathbf{p}}(f)$, 
where $f$ is of type $\boldsymbol{\lambda}$, described in Subsection~\ref{s2.4} and the classical
branching rule for the symmetric group, see \cite[Section~2.8]{Sa}, one immediately recovers 
the branching rule for the restriction from $S(\ell,d)$ to $S(\ell,d-1)$ 
as described in \cite{Mar}. Namely, the restriction of $\mathrm{L}_{\mathbf{p}}$
to $S(\ell,d-1)$ is a multiplicity free direct sum of $\mathrm{L}_{\mathbf{q}}$ where
$\mathbf{q}$ is obtained from $\mathbf{p}$ by removing one removable node from one of the
parts of $\mathbf{p}$.

\subsection{Gelfand model}\label{s2.8}

Recall that a {\em Gelfand model} for a finite group $G$ is a $G$-module isomorphic 
to a multiplicity-free direct sum of all simple $G$-modules. Similarly one defines Gelfand 
models for semi-simple algebras. Let $\mathcal{I}$ be the set of all involutions in $S_d$, 
that is all elements $w\in S_d$ satisfying $w^2=e$. Define an $S_d$-module structure on 
$\mathbb{C}[\mathcal{I}]$, for $\sigma\in S_d$ and $w\in \mathcal{I}$, as follows:
\begin{gather*}
\sigma\cdot w=(-1)^{\mathrm{inv}(\sigma,w)}(\sigma w \sigma^{-1}),\quad\text{ where }\\
\mathrm{inv}(\sigma,w):=|\{(i,j)\mid i,j\in\underline{d},\, i<j,\,w(i)=j,\,\sigma(i)>\sigma(j)\}|.
\end{gather*}

\begin{proposition}{\rm (\cite{IRS}, \cite{APR1})}
The $S_d$-module $\mathbb{C}[\mathcal{I}]$ is a Gelfand model for $S_d$.
\end{proposition}
This model (sometimes referred to as the {\em involutive Gelfand model}) was 
generalized to wreath products in \cite{APR2,CF}, to inverse semigroups in \cite{KM2} 
and to general diagram algebras in \cite{HRe}, \cite{Ma}, see also references in these paper for 
other generalizations. An alternative approach to Gelfand models for certain classes of
groups can be found in \cite{CM}.

In our setup it is fairly straightforward to combine the above model with the construction used 
in \cite{KM2}, \cite{Ma} to produce a Gelfand model for $\cG_{(\ell,d)}$ (significantly simplifying 
arguments from \cite{APR2}). For each $f\in\cG_{(\ell,d)}$ denote by $\mathcal{I}^f$ the set 
of all involutions in $\cG_{(\ell,d)}(f,f)$. Define a representation 
$\mathrm{Gelfand}$ of $\mathbb{C}\cG_{(\ell,d)}$ as follows:
\begin{itemize}
\item Set $\mathrm{Gelfand}(f):=\mathbb{C}[\mathcal{I}^f]$.
\item For $f,g\in \cG_{(\ell,d)}$, $\sigma\in \cG_{(\ell,d)}(f,g)$ and $w\in \mathcal{I}^f$ set
\begin{displaymath}
\mathrm{Gelfand}(\sigma)\cdot w=(-1)^{\mathrm{inv}(\sigma,w)}(\sigma w \sigma^{-1}). 
\end{displaymath}
\item Extend this to the whole of $\mathbb{C}\cG_{(\ell,d)}$ by linearity.
\end{itemize}

\begin{corollary}\label{prop33}
The $\mathbb{C}\cG_{(\ell,d)}$-module $\mathrm{Gelfand}$ is a Gelfand model for $\mathbb{C}\cG_{(\ell,d)}$.
\end{corollary}

\begin{proof}
The fact that  $\mathrm{Gelfand}$ is a $\mathbb{C}\cG_{(\ell,d)}$-module follows directly from 
our definitions and the construction of the Gelfand model for $S_d$ in \cite{APR1}. Taking 
into account the classification of simple $\mathbb{C}\cG_{(\ell,d)}$-modules in 
Subsection~\ref{s2.4}, to prove that $\mathrm{Gelfand}$ is a Gelfand model, we
need to prove that for each $f\in\cG_{(\ell,d)}$ the space $\mathrm{Gelfand}(f)$ is a Gelfand 
model for $\cG_{(\ell,d)}(f,f)$. This again follows directly from the definitions and the 
main result of \cite{APR1}.
\end{proof}

\section{Schur-Weyl dualities for $S(\ell,d)$}\label{s3}

\subsection{Classical Schur-Weyl duality}\label{s3.1}

For $n\in\mathbb{Z}_{>0}$, consider the (infinite!) group 
$\mathbf{GL}_n=\mathbf{GL}_n(\mathbb{C})$ and its {\em natural} representation 
$V:=\mathbb{C}^{n}$ with standard basis $\mathbf{v}:=(v_1,v_2,\dots,v_n)$. 
For $d\in\mathbb{Z}_{>0}$, consider the $d$-th tensor power 
$V^{\otimes d}$ with the usual diagonal coproduct action of $\mathbf{GL}_n$. 
The symmetric group $S_d$ acts on $V^{\otimes d}$ by permuting the components of the 
tensor product. This action clearly commutes with the action of $\mathbf{GL}_n$. 
Moreover, these two actions have the {\em double centralizer property} in the sense 
that every linear operator on $V^{\otimes d}$ which commutes with the action of 
$\mathbb{C}[\mathbf{GL}_n]$ is given by the action of $\mathbb{C}[S_d]$ and vice versa, 
depicted as follows:
\begin{equation}
\label{SchurWeyl}
\xymatrix{
V^{\otimes d} \ar@(dr,r)[]_{S_d} \ar@(dl,l)[]^{\mathrm{GL}_n} 
} 
\end{equation}
This is the classical {\em Schur-Weyl duality} from \cite{Schur1}, \cite{Schur2}, \cite{Weyl}.

The action of $\mathbb{C}[\mathbf{GL}_n]$ on $V^{\otimes d}$ is certainly never faithful.
The action of $\mathbb{C}[S_d]$ on $V^{\otimes d}$ is faithful if and only if $n\geq d$.
If $n<d$, then the kernel of this action is given by the ideal in $\mathbb{C}[S_d]$ 
corresponding to all Specht $S_d$-modules $\mathscr{S}_{\mu}$, where $\mu\vdash d$ has 
more than $n$ rows, see \cite[Theorem~9.1.2]{GW}.

\subsection{Splitting the left action}\label{s3.2}

Let now $l,n\in\mathbb{Z}_{>0}$ with  $l\leq n$. Fix a composition 
$\mathbf{k}=(k_1,k_2,\dots,k_\ell)\in\Lambda(l,n)$ in which  all $k_i>0$. 
Consider the block-diagonal subgroup $\mathbf{GL}_{\mathbf{k}}\cong 
\mathbf{GL}_{k_1}\times\mathbf{GL}_{k_2}\times\dots\times\mathbf{GL}_{k_\ell}$
in $\mathbf{GL}_{n}$ given by all matrices of the following form:
\begin{center}
\begin{picture}(100,100)
\drawline(10.00,90.00)(90.00,90.00)
\drawline(90.00,10.00)(90.00,90.00)
\drawline(90.00,10.00)(10.00,10.00)
\drawline(10.00,90.00)(10.00,10.00)
\drawline(90.00,30.00)(10.00,30.00)
\drawline(90.00,50.00)(10.00,50.00)
\drawline(90.00,70.00)(10.00,70.00)
\drawline(30.00,90.00)(30.00,10.00)
\drawline(50.00,90.00)(50.00,10.00)
\drawline(70.00,90.00)(70.00,10.00)
\put(21,80){\makebox(0,0)[cc]{\tiny $\mathbf{GL}_{k_1}$}}
\put(41,60){\makebox(0,0)[cc]{\tiny $\mathbf{GL}_{k_2}$}}
\put(81,20){\makebox(0,0)[cc]{\tiny $\mathbf{GL}_{k_\ell}$}}
\put(40,80){\makebox(0,0)[cc]{\tiny $0$}}
\put(60,80){\makebox(0,0)[cc]{\tiny $\cdots$}}
\put(80,80){\makebox(0,0)[cc]{\tiny $0$}}
\put(20,60){\makebox(0,0)[cc]{\tiny $0$}}
\put(60,60){\makebox(0,0)[cc]{\tiny $\cdots$}}
\put(80,60){\makebox(0,0)[cc]{\tiny $0$}}
\put(20,42){\makebox(0,0)[cc]{\tiny $\vdots$}}
\put(40,42){\makebox(0,0)[cc]{\tiny $\vdots$}}
\put(60,42){\makebox(0,0)[cc]{\tiny $\ddots$}}
\put(80,42){\makebox(0,0)[cc]{\tiny $\vdots$}}
\put(20,20){\makebox(0,0)[cc]{\tiny $0$}}
\put(40,20){\makebox(0,0)[cc]{\tiny $0$}}
\put(60,20){\makebox(0,0)[cc]{\tiny $\cdots$}}
\end{picture}
\end{center}

By restriction, the $\mathbf{GL}_{n}$-module $V$ from the previous subsection 
becomes a $\mathbf{GL}_{\mathbf{k}}$-module and $V$ decomposes as  $V=V_1\oplus V_2\oplus\dots\oplus V_\ell$, where $V_i$ the subspace of  $V$ spanned by all $v_j$ where
\begin{displaymath}
j\in\{k_1+k_2+\dots+k_{i-1}+1,k_1+k_2+\dots+k_{i-1}+2,\dots,k_1+k_2+\dots+k_{i}\}. 
\end{displaymath}
In particular we have for $i,j\in \{1,2,\dots,\ell\}$ that $V_i$ is the natural $\mathbf{GL}_{k_i}$-module while the action
of $\mathbf{GL}_{k_i}$ on $V_j$ is trivial whenever $j\neq i$.

\subsection{$\cG_{(\ell,d)}$-action on $V^{\otimes d}$}\label{s3.3}

The space $V^{\otimes d}$ has the structure of a 
$\mathbb{C}\cG_{(\ell,d)}$-module 
$\mathrm{G}$ defined as follows:
\begin{itemize}
\item For $f\in \cG_{(\ell,d)}$, we set $\mathrm{G}(f):=V_{f(1)}\otimes 
V_{f(2)}\otimes\dots\otimes V_{f(d)}$.
\item For $f,g\in \cG_{(\ell,d)}$ of the same type and $\sigma\in \cG_{(\ell,d)}(f,g)$, the linear 
map $\mathrm{G}(\sigma)$ acts by permuting factors of the tensor product, namely,
\begin{equation}\label{eqnn2}
\sigma(w_1\otimes w_2\otimes\dots\otimes w_d):=
w_{\sigma^{-1}(1)}\otimes w_{\sigma^{-1}(2)}\otimes\dots\otimes w_{\sigma^{-1}(d)}
\end{equation}
where $w_i\in V_{f(i)}$ for all $i$ (note that permutation of components induces the
{\em opposite} action on indices of the components, which explains the appearance
of $\sigma^{-1}$ in \eqref{eqnn2}).
\end{itemize}
It is straightforward to check that this gives a well-defined $\mathbb{C}\cG_{(\ell,d)}$-module.
Using the equivalence \eqref{eq1}, this defines on $V^{\otimes d}$ the structure of an
$A_{(\ell,d)}$-module.

\begin{lemma}\label{lem11}
{\tiny\hspace{2mm}}

\begin{enumerate}[$($i$)$]
\item\label{lem11.1} The action of $\mathbf{GL}_{\mathbf{k}}$ preserves $\mathrm{G}(f)$
for each  $f\in\mathbb{C}\cG_{(\ell,d)}$.
\item\label{lem11.2} The action of $A_{(\ell,d)}$ on $V^{\otimes d}$ commutes with the action of $\mathbf{GL}_{\mathbf{k}}$.
\end{enumerate}
\end{lemma}

\begin{proof}
Let $w_1\otimes w_2\otimes\dots\otimes w_d\in V_{f(1)}\otimes V_{f(2)}\otimes\dots\otimes
V_{f(d)}$. Choose any  $x_i\in \mathbf{GL}_{k_i}$ for all $i$ and let 
$x=\mathrm{diag}(x_1,x_2,\dots,x_\ell)$ be the corresponding element in 
$\mathbf{GL}_{\mathbf{k}}$. Then, using the definitions and the fact that
$\mathbf{GL}_{k_i}$ acts trivially on $V_j$ for $i\neq j$, we have
\begin{equation}\label{eqnn3}
x\cdot (w_1\otimes w_2\otimes \dots\otimes w_d)=
(x_{f(1)}\cdot w_1)\otimes (x_{f(2)}\cdot w_2)\otimes \cdots \otimes (x_{f(d)}\cdot w_d).
\end{equation}
It follows that the action of $\mathbf{GL}_{\mathbf{k}}$
preserves each $\mathrm{G}(f)$. This proves claim~\eqref{lem11.1}.
Moreover, this also implies that to prove claim~\eqref{lem11.2} 
it is enough to show that the action of 
$\mathbf{GL}_{\mathbf{k}}$ commutes with the action of each $\sigma\in \cG_{(\ell,d)}(f,g)$.

Applying $\sigma$ to \eqref{eqnn3}, we get
\begin{multline}\label{eqnn1}
\sigma\circ x\cdot (w_1\otimes w_2\otimes \dots\otimes w_d)=\\=
(x_{f(\sigma^{-1}(1))}\cdot w_{\sigma^{-1}(1)})\otimes (x_{f(\sigma^{-1}(2))}\cdot 
w_{\sigma^{-1}(2)})\otimes \cdots \otimes (x_{f(\sigma^{-1}(d))}\cdot w_{\sigma^{-1}(d)}).
\end{multline}
Similarly to \eqref{eqnn3}, acting by $x$ on \eqref{eqnn2}, we get 
\begin{multline*}
x\circ\sigma\cdot (w_1\otimes w_2\otimes \dots\otimes w_d)=\\=
(x_{g(1)}\cdot w_{\sigma^{-1}(1)})\otimes (x_{g(2)}\cdot w_{\sigma^{-1}(2)})
\otimes \cdots \otimes (x_{g(d)}\cdot w_{\sigma^{-1}(d)}).
\end{multline*}
The latter coincides with \eqref{eqnn1} since $f(i)=g(\sigma(i))$ for all $i$ and therefore 
also $f(\sigma^{-1}(i))=g(i)$ for all $i$. This proves claim~\eqref{lem11.2} and thus
completes the proof.
\end{proof}

\subsection{Schur-Weyl duality for $S(\ell,d)$}\label{s3.4}

The Hecke-algebra version of the next theorem appears in \cite{Hu}, \cite{SS}, \cite{Sh}, see 
also further cases in \cite{Re} and \cite{ATY}.

\begin{theorem}\label{thm12}
The actions of $A_{(\ell,d)}$ and $\mathbf{GL}_{\mathbf{k}}$ on $V^{\otimes d}$ have the 
double centralizer property 
\begin{equation}
\label{doublecentr}
\xymatrix{
V^{\otimes d} \ar@(dr,r)[]_{A_{(\ell,d)}} \ar@(dl,l)[]^{\mathrm{GL}_{\mathbf{k}}} 
} 
\end{equation}
in the sense that they generate each others centralizers. 
\end{theorem}

\begin{proof}
We consider the action of $\mathbb{C}\cG_{(\ell,d)}$ instead of the action of $A_{(\ell,d)}$.
By Lemma~\ref{lem11}\eqref{lem11.1}, each $\mathrm{G}(f)$ is invariant under the action 
of $\mathbf{GL}_{\mathbf{k}}$. We claim that $A_{(\ell,d)}$ surjects onto the space of 
$\mathbf{GL}_{\mathbf{k}}$-intertwiners between the $\mathbf{GL}_{\mathbf{k}}$-modules
$\mathrm{G}(f)$ and $\mathrm{G}(g)$, where $f$ and $g$ are of the same type 
$\boldsymbol{\lambda}$. Indeed, the action of $\mathbf{GL}_{\mathbf{k}}$ on both 
$\mathrm{G}(f)$ and $\mathrm{G}(g)$ can be computed using \eqref{eqnn3}. It follows that,
as $\mathbf{GL}_{\mathbf{k}}$-modules, both $\mathrm{G}(f)$ and $\mathrm{G}(g)$
can be identified with the external tensor product of the $\lambda_i$-th tensor powers
$V_i^{\otimes \lambda_i}$ of the natural $\mathbf{GL}_{k_i}$-modules $V_i$, where 
$i=1,2,\dots,\ell$. By construction, the ``local'' action of $\mathbb{C}\cG_{(\ell,d)}$ on 
$V_i^{\otimes \lambda_i}$ is given by $S_{\lambda_i}$ which acts by permuting components
of the tensor product. This is exactly the setup of the classical Schur-Weyl duality \eqref{SchurWeyl}
and hence it follows that, locally, the  action of  $\mathbb{C}\cG_{(\ell,d)}$ does generate
the whole centralizer of the $\mathbf{GL}_{k_i}$-module $V_i^{\otimes \lambda_i}$.
Note that this local action of $\mathbb{C}\cG_{(\ell,d)}$ is independent of all
other components.

Since the (external) tensor product of simple finite dimensional $\mathbb{C}$-algebras is a 
simple $\mathbb{C}$-algebra (as $\mathbb{C}$ is algebraically closed), 
by comparing dimension it follows that $\mathbb{C}\cG_{(\ell,d)}$
does generate all $\mathbf{GL}_{\mathbf{k}}$-intertwiners between
$\mathrm{G}(f)$ and $\mathrm{G}(g)$. Summing up over all $f$ and $g$ we get that
the action of $A_{(\ell,d)}$ generates the centralizer of the
$\mathbf{GL}_{\mathbf{k}}$-action on $V^{\otimes d}$.

As $\mathbf{GL}_{\mathbf{k}}$ is reductive, its action on $V^{\otimes d}$ is 
semi-simple. The group algebra of the finite group $A_{(\ell,d)}$ is clearly semisimple
and is the centralizer of the action of $\mathbf{GL}_{\mathbf{k}}$ by the above. 
Therefore the desired double centralizer property follows from the Double
Centralizer Theorem, see \cite[Subsection~3.2]{KP}.
\end{proof}

\subsection{An extremal example: the symmetric inverse semigroup}\label{s3.5}

Recall that, for $d\in\mathbb{Z}_{\geq 0}$, the symmetric inverse semigroup $IS_d$
(a.k.a. the rook monoid $R_d$)
is the monoid of all bijections between subsets of $\underline{d}$, see 
\cite[Section~2.5]{GM} for details. The symmetric group $S_d$ is the group of invertible 
elements in $IS_d$. The monoid $IS_d$ is generated by $S_d$ and the (idempotent) identity
transformation $\varepsilon_1$ on the subset $\{2,3,\dots,n\}$ of $\underline{d}$
(for a ``Coxeter-like'' presentation of $IS_d$ we refer the reader to \cite[Remark~4.13]{KM}).

If $l=2$, $k_1=n-1$ and $k_2=1$, then the action on the left hand side of the Schur-Weyl 
duality from Theorem~\ref{thm12} can be viewed as the action of $\mathbf{GL}_{n-1}$, 
with the trivial action on $V_2$. This is exactly the setup of Solomon's version of 
Schur-Weyl duality for $IS_d$  established in \cite{So}. 
The action of $\mathbb{C}[IS_d]$ is faithful for $d<n$.

Put together with Theorem~\ref{thm7}, this gives a surjection 
$\mathbb{C}[S(2,d)]\tto\mathbb{C}[IS_d]$ which was already observed on the level of the 
corresponding Hecke algebras in \cite{HR}. For convenience, we give here an explicit formula 
for this surjection in terms of Coxeter generators of the type $B$ Weyl group $S(2,d)$:

\begin{lemma}\label{lem15}
There is a unique epimorphism of algebras $\Psi:\mathbb{C}[S(2,d)]\tto\mathbb{C}[IS_d]$ such that
\begin{displaymath}
\Psi(s_i)=s_i,\quad i=1,2,\dots,d-1;\qquad\qquad \Psi(s_0)=2\varepsilon_1-e. 
\end{displaymath}
\end{lemma}

\begin{proof}
Surjectivity is directly clear as $\mathbb{C}[IS_d]$ is generated by $S_d$ and $\varepsilon_1$. 
To verify that this is a homomorphism, the only non-trivial relations to check are those 
involving the image of $s_0$ with itself and with the image of $s_1$. For the first one, 
using the fact that $\varepsilon_1$ is an idempotent, we have
\begin{displaymath}
(2\varepsilon_1-e)^2=4\varepsilon_1^2-4\varepsilon_1+e=4\varepsilon_1-4\varepsilon_1+e=e.
\end{displaymath}
It remains to check that 
$(2\varepsilon_1-e)s_1(2\varepsilon_1-e)s_1=s_1(2\varepsilon_1-e)s_1(2\varepsilon_1-e)$.
Opening the brackets and canceling the obvious equal summands, this reduces 
(up to a non-zero scalar) to a desired equality  $\varepsilon_1s_1\varepsilon_1s_1=s_1\varepsilon_1s_1\varepsilon_1$. 
It is straightforward to check that both sides of the latter are, in fact, equal 
to the identity transformation on the subset $\{3,4,\dots,d\}$ of $\underline{d}$.
\end{proof}

The homomorphism $\Psi$ allows us to view simple $\mathbb{C}[IS_d]$-modules
as simple $\mathbb{C}[S(2,d)]$-modules. Combining Solomon's version of the
Schur-Weyl duality mentioned above and the results of Subsection~\ref{s3.7} below,
one gets that the simple $\mathbb{C}[S(2,d)]$-modules obtained in this way
are exactly the modules $L_{(\mu_1,\mu_2)}$, where the partition $\mu_2$ has one part.
This parametrization is used, for instance, in \cite{ES2}.
In particular, our remarks on the basis in simple modules from the end of 
Subsection~\ref{s2.4} correspond in this case to the main result in \cite{Gr}.

\subsection{The kernel of the $A_{(\ell,d)}$-action}\label{s3.7}

The action of $\mathbf{GL}_{\mathbf{k}}$ on $V^{\otimes d}$ is, of course, never faithful.
The action of $A_{(\ell,d)}$ on $V^{\otimes d}$ is usually not faithful either. The kernel of the
latter action (or rather of the action of $\mathbb{C}\cG_{(\ell,d)}$) can be described in 
terms similar to the description of the kernel of the action of $\mathbb{C}[S_d]$ in the 
classical Schur-Weyl duality.

\begin{lemma}\label{lem17}
The kernel of the action of  $\mathbb{C}\cG_{(\ell,d)}$ on $V^{\otimes d}$ is given by the  
ideal in $\mathbb{C}\cG_{(\ell,d)}$ corresponding to all $\mathbb{C}\cG_{(\ell,d)}$-modules
$\mathrm{L}_{\mathbf{p}}$ for $\mathbf{p}\in \mathbf{T}_{\boldsymbol{\lambda}}$ 
with $\boldsymbol{\lambda}\in\Lambda(\ell,d)$, such that for some $i\in\{1,2,\dots,\ell\}$ 
the partition $p_i\vdash\lambda_i$ has more than $k_i$  rows.
\end{lemma}

\begin{proof}
It is enough to determine the part of the kernel inside 
$\mathbb{C}\cG_{(\ell,d)}(f_{\boldsymbol{\lambda}},f_{\boldsymbol{\lambda}})$.
The latter is an algebra acting on $\mathrm{G}(f_{\boldsymbol{\lambda}})$ and for it the
statement follows from the classical Schur-Weyl duality, see Subsection~\ref{s3.1}.
\end{proof}

For example, in the case $\mathbf{G}=\mathbf{GL}_1\times\mathbf{GL}_1\times\dots\times
\mathbf{GL}_1$, all simple $\mathbb{C}\cG_{(\ell,d)}$-modules which are not
indexed by $1$-row multi-partitions are killed.

\begin{remark}\label{remnnew}
{\rm
An important example here is the case where $\ell=2$ and where we have the action of $\mathbf{G}:= \mathbf{GL}_{n}\times\mathbf{GL}_{n}\subset\mathbf{GL}_{2n}$,
Then the group algebra $\mathbb{C}\cG_{(2,d)}$ of the Weyl group of type $B_d$ acts faithfully as endomorphisms of  $V^{\otimes d}$ as long as $n\geq d$.  A quantized version of this special case appears in \cite{SaSt}, see also \cite{Wang}, \cite{ES1}. In \cite{ES1} in fact a quantized version of Theorem~\ref{thm95} below appears. (Note that in there it is shown that the action of the coideal subalgebra, denoted $\mathcal{H}$ in \cite{ES1}, on the tensor space $V^{\otimes d}$ agrees with the action of the quantum group attached to the Lie algebra of $\mathbf{G}$.)
}
\end{remark}

\subsection{Wreath product with an arbitrary abelian group}\label{s3.8}

Let $A$ be an abelian group with $\ell$ elements. Then we have an isomorphism 
$\mathbb{C}[A]\cong \mathbb{C}[\mathbf{C}_\ell]$. Choose some $\mathbf{k}=(k_1,k_2,\dots,k_\ell)$ 
such that all $k_i>d$. Then the endomorphism algebra of $V$ from Subsection~\ref{s3.2} 
is isomorphic to both $\mathbb{C}[A]$ and $\mathbb{C}[\mathbf{C}_\ell]$. Going to the other 
side of the Schur-Weyl duality in Theorem~\ref{thm12} and taking Lemma~\ref{lem17} into 
account, we obtain $\mathbb{C}[A\wr S_d]\cong\mathbb{C}[\mathbf{C}_\ell\wr S_d]$. 
Therefore all results of this paper can be reformulated 
(for an appropriately defined action of $A$) for the wreath product $A\wr S_d$.

\section{Extension to $G(\ell,k,d)$}\label{s4}

\subsection{Complex reflection groups $G(\ell,k,d)$}\label{s4.1}
For the rest of the article fix  $k\in\mathbb{Z}_{>0}$ such that $k|\ell$.  
The realization of $S(\ell,d)$ as the group of $d\times d$ matrices as described 
in Subsection~\ref{s2.1} contains an index $k$  subgroup given by all matrices in $S(\ell,d)$ with determinant 
in $\mathbf{C}_{l/k}$, It is the complex reflection group $G(\ell,k,d)$. In this section we generalize most of the above results to the case of $G(\ell,k,d)$. Our approach is motivated by \cite{HS}.

\subsection{The quotient groupoid $\cG_{(\ell,d)}^k$}\label{s4.2}

Let $\theta:\underline{\ell}\to\underline{\ell}$ denote the permutation sending $i$ to $i+1$ 
for $i<\ell$ and $\ell$ to $1$. Then $\theta$ is a cycle of order $\ell$. Consider the permutation $\theta_k:=\theta^{\ell/k}$ and let $H_k$ be the group generated by $\theta_k$, in particular $|H_k|=k$, the order of $\theta_k$. 

The group $H_k$ acts by automorphisms on $\cG_{(\ell,d)}$ permuting the colors. 
More explicitly, this action is described as follow:
\begin{itemize}
\item For $f\in \cG_{(\ell,d)}$ with $f=(f(1),f(2),\dots,f(d))$, we have 
\begin{displaymath}
\theta_k(f)=(\theta_k(f(1)),\theta_k(f(2)),\dots,\theta_k(f(d))). 
\end{displaymath}
\item For $f,g\in \cG_{(\ell,d)}$ and $\sigma\in\cG_{(\ell,d)}(f,g)$ we have $\theta_k(\sigma)=\sigma$.
\end{itemize}
Note that the action of $H_k$ on $\underline{\ell}$ is {\em free} in the sense that the 
stabilizer of each element is trivial. This implies 
that the action of $H_k$ on $\cG_{(\ell,d)}$ is free as well.

Denote by $\cG_{(\ell,d)}^k$ the {\em quotient groupoid} $\cG_{(\ell,d)}/H_k$ defined as follows:
\begin{itemize}
\item Objects of $\cG_{(\ell,d)}^k$ are orbits of the action of $H_k$ on objects of $\cG_{(\ell,d)}$.
For $f\in \cG_{(\ell,d)}$, we will denote the $H_k$-orbit of $f$ by $f^{(k)}$.
\item For $f,g\in \cG_{(\ell,d)}$, elements in the set $\cG_{(\ell,d)}^k(f^{(k)},g^{(k)})$ are 
orbits of $H_k$ on the $H_k$-invariant set
\begin{displaymath}
\coprod_{f'\in f^{(k)},g'\in g^{(k)}} \cG_{(\ell,d)}(f',g').
\end{displaymath}
For $\sigma\in \cG_{(\ell,d)}(f,g)$, we will denote the $H_k$-orbit of $\sigma$ by $\sigma^{(k)}$.
\item For $f\in \cG_{(\ell,d)}$, the identity morphism in $\cG_{(\ell,d)}^k(f^{(k)},f^{(k)})$ is  $e_f^{(k)}$.
\item The composition in $\cG_{(\ell,d)}^k$ is the induced composition from $\cG_{(\ell,d)}$.
\end{itemize}

\begin{example}\label{ex49}
{\rm
Let $d=\ell=2$ (the setup of Example~\ref{ex1}) and take $k=2$. Then $\theta=\theta_k$ 
swaps the two colors. The quotient groupoid $\cG_{(\ell,d)}^k$ has two objects:
\begin{displaymath}
o_1=
\{\begin{picture}(30,5)
\put(05,02.50){\makebox(0,0)[cc]{\color{red}{$\circ$}}}
\put(25,02.50){\makebox(0,0)[cc]{\color{red}{$\circ$}}}
\end{picture}\qquad\qquad
\begin{picture}(30,5)
\put(05,02.50){\makebox(0,0)[cc]{\color{blue}{$\bullet$}}}
\put(25,02.50){\makebox(0,0)[cc]{\color{blue}{$\bullet$}}}
\end{picture}\}\qquad\text{ and }\qquad
o_2=
\{
\begin{picture}(30,5)
\put(05,02.50){\makebox(0,0)[cc]{\color{red}{$\circ$}}}
\put(25,02.50){\makebox(0,0)[cc]{\color{blue}{$\bullet$}}}
\end{picture}\qquad\qquad
\begin{picture}(30,5)
\put(05,02.50){\makebox(0,0)[cc]{\color{blue}{$\bullet$}}}
\put(25,02.50){\makebox(0,0)[cc]{\color{red}{$\circ$}}}
\end{picture}\}
\end{displaymath}
There are no morphisms between these objects. Both endomorphism sets have two elements. For the object $o_1$ we have
endomorphisms
\begin{displaymath}
\left\{\begin{picture}(80,20)
\put(05,02.50){\makebox(0,0)[cc]{\color{red}{$\circ$}}}
\put(25,02.50){\makebox(0,0)[cc]{\color{red}{$\circ$}}}
\put(05,17.50){\makebox(0,0)[cc]{\color{red}{$\circ$}}}
\put(25,17.50){\makebox(0,0)[cc]{\color{red}{$\circ$}}}
\put(55,02.50){\makebox(0,0)[cc]{\color{blue}{$\bullet$}}}
\put(75,02.50){\makebox(0,0)[cc]{\color{blue}{$\bullet$}}}
\put(55,17.50){\makebox(0,0)[cc]{\color{blue}{$\bullet$}}}
\put(75,17.50){\makebox(0,0)[cc]{\color{blue}{$\bullet$}}}
\color{red}{\dottedline{1}(05.00,02.50)(05.00,17.50)
\dottedline{1}(25.00,02.50)(25.00,17.50)}
\color{blue}{\drawline(51.50,02.50)(51.50,17.50)
\drawline(71.50,02.50)(71.50,17.50)}
\end{picture}\right\}
\qquad\qquad\text{ and }\qquad\qquad
\left\{\begin{picture}(80,20)
\put(05,02.50){\makebox(0,0)[cc]{\color{red}{$\circ$}}}
\put(25,02.50){\makebox(0,0)[cc]{\color{red}{$\circ$}}}
\put(05,17.50){\makebox(0,0)[cc]{\color{red}{$\circ$}}}
\put(25,17.50){\makebox(0,0)[cc]{\color{red}{$\circ$}}}
\put(55,02.50){\makebox(0,0)[cc]{\color{blue}{$\bullet$}}}
\put(75,02.50){\makebox(0,0)[cc]{\color{blue}{$\bullet$}}}
\put(55,17.50){\makebox(0,0)[cc]{\color{blue}{$\bullet$}}}
\put(75,17.50){\makebox(0,0)[cc]{\color{blue}{$\bullet$}}}
\color{red}{\dottedline{1}(05.00,02.50)(25.00,17.50)
\dottedline{1}(25.00,02.50)(05.00,17.50)}
\color{blue}{\drawline(51.50,02.50)(71.50,17.50)
\drawline(71.50,02.50)(51.50,17.50)}
\end{picture}\right\}
\end{displaymath}
which form a group isomorphic to $S_2$. For the object $o_2$ we have endomorphisms
\begin{displaymath}
\left\{\begin{picture}(80,20)
\put(05,02.50){\makebox(0,0)[cc]{\color{red}{$\circ$}}}
\put(75,02.50){\makebox(0,0)[cc]{\color{red}{$\circ$}}}
\put(05,17.50){\makebox(0,0)[cc]{\color{red}{$\circ$}}}
\put(75,17.50){\makebox(0,0)[cc]{\color{red}{$\circ$}}}
\put(55,02.50){\makebox(0,0)[cc]{\color{blue}{$\bullet$}}}
\put(25,02.50){\makebox(0,0)[cc]{\color{blue}{$\bullet$}}}
\put(55,17.50){\makebox(0,0)[cc]{\color{blue}{$\bullet$}}}
\put(25,17.50){\makebox(0,0)[cc]{\color{blue}{$\bullet$}}}
\color{red}{\dottedline{1}(05.00,02.50)(05.00,17.50)
\dottedline{1}(75.00,02.50)(75.00,17.50)}
\color{blue}{\drawline(51.50,02.50)(51.50,17.50)
\drawline(22.00,02.50)(22.00,17.50)}
\end{picture}\right\}
\qquad\qquad\text{ and }\qquad\qquad
\left\{\begin{picture}(80,20)
\put(55,02.50){\makebox(0,0)[cc]{\color{red}{$\circ$}}}
\put(25,02.50){\makebox(0,0)[cc]{\color{red}{$\circ$}}}
\put(05,17.50){\makebox(0,0)[cc]{\color{red}{$\circ$}}}
\put(75,17.50){\makebox(0,0)[cc]{\color{red}{$\circ$}}}
\put(05,02.50){\makebox(0,0)[cc]{\color{blue}{$\bullet$}}}
\put(75,02.50){\makebox(0,0)[cc]{\color{blue}{$\bullet$}}}
\put(55,17.50){\makebox(0,0)[cc]{\color{blue}{$\bullet$}}}
\put(25,17.50){\makebox(0,0)[cc]{\color{blue}{$\bullet$}}}
\color{red}{\dottedline{1}(55.00,02.50)(75.00,17.50)
\dottedline{1}(25.00,02.50)(05.00,17.50)}
\color{blue}{\drawline(02.50,02.50)(22.50,17.50)
\drawline(71.50,02.50)(51.50,17.50)}
\end{picture}\right\}
\end{displaymath}
which form a group isomorphic to $H_2=\langle\theta\rangle$ 
(which is also isomorphic to $S_2$ by a coincidence).
}
\end{example}

We denote by $\mathbb{C}\cG_{(\ell,d)}^k$ the linearization of $\cG_{(\ell,d)}^k$ constructed 
similarly to the construction of $\mathbb{C}\cG_{(\ell,d)}$ from $\cG_{(\ell,d)}$ in Subsection~\ref{s2.3}.

\subsection{Endomorphisms in $\cG_{(\ell,d)}^k$}\label{s4.3}

The group $H_k$ acts naturally on the set $\Lambda(\ell,d)$ of all types 
by permuting the indices (i.e. colors) of a type. We consider the set 
$\Lambda(\ell,d)/H_k$ of the corresponding orbits. For 
$\boldsymbol{\lambda}\in \Lambda(\ell,d)$, we denote by 
$H_k^{\boldsymbol{\lambda}}$ the stabilizer of $\boldsymbol{\lambda}$ in $H_k$ 
and by $\boldsymbol{\lambda}^{(k)}$ the orbit of $\boldsymbol{\lambda}$ with respect 
to the action of $H_k$. Thus, in Example~\ref{ex49}, the stabilizer of the type $(2,0)$ 
is trivial and the orbit of this type contains one more element, namely $(0,2)$; while 
the stabilizer of the type $(1,1)$  coincides with $H_2$ and the orbit of $(1,1)$ 
contains no other elements.

As connected components of $\cG_{(\ell,d)}$ are in bijection with $\Lambda(\ell,d)$, 
connected components of $\cG_{(\ell,d)}^k$ are in bijection with $\Lambda(\ell,d)/H_k$. 
To determine simple $\mathbb{C}\cG_{(\ell,d)}^k$-modules, we need to determine the endomorphism 
groups of objects in $\cG_{(\ell,d)}^k$.

Let $\boldsymbol{\lambda}\in \Lambda(\ell,d)$ and $f\in \cG_{(\ell,d)}$ be of type 
$\boldsymbol{\lambda}$. The assignment $\sigma\to \sigma^{(k)}$ defines a homomorphism from 
$\cG_{(\ell,d)}(f,f)$ to $\cG_{(\ell,d)}^k(f^{(k)},f^{(k)})$ which is injective, since the action of $H_k^{\boldsymbol{\lambda}}$ is free. Hence, we can view $\cG_{(\ell,d)}(f,f)$ as a subgroup of $\cG_{(\ell,d)}^k(f^{(k)},f^{(k)})$.

For a generator $h\in H_k^{\boldsymbol{\lambda}}$ consider the cyclic subgroup of 
$\cG_{(\ell,d)}^k(f^{(k)},f^{(k)})$ generated by $\sigma^{(k)}_{(f,h(f))}$. Since 
the action of $H_k^{\boldsymbol{\lambda}}$ is free, all orbits of $h$ have
order $|H_k^{\boldsymbol{\lambda}}|$. It follows that the order of $\sigma^{(k)}_{(f,h(f))}$
equals $|H_k^{\boldsymbol{\lambda}}|$ and hence  $H_k^{\boldsymbol{\lambda}}$ can be identified with the subgroup generated by 
$\sigma^{(k)}_{(f,h(f))}$ by 
mapping $h\in  H_k^{\boldsymbol{\lambda}}$ to $\sigma^{(k)}_{(f,h(f))}$.

\begin{lemma}\label{lem51}
Let $\boldsymbol{\lambda}\in \Lambda(\ell,d)$ and $f\in \cG_{(\ell,d)}$ be of 
type $\boldsymbol{\lambda}$.  Then 
\begin{equation}\label{eq9}
\cG_{(\ell,d)}^k(f^{(k)},f^{(k)})\cong H_k^{\boldsymbol{\lambda}}\ltimes \cG_{(\ell,d)}(f,f).
\end{equation}
\end{lemma}

\begin{proof}
That both $H_k^{\boldsymbol{\lambda}}$ and $\cG_{(\ell,d)}(f,f)$ are subgroups 
of $\cG_{(\ell,d)}^k(f^{(k)},f^{(k)})$ with trivial intersection follows from the construction. It is easy to check that $\cG_{(\ell,d)}(f,f)$ is normal. It remains to show 
that both sides of \eqref{eq9} have the same cardinality. Clearly,
\begin{eqnarray}\label{eq91}
|H_k^{\boldsymbol{\lambda}}\ltimes \cG_{(\ell,d)}(f,f)|
&=&|H_k^{\boldsymbol{\lambda}}|\cdot |\cG_{(\ell,d)}(f,f)|. 
\end{eqnarray}
From the definition $\cG_{(\ell,d)}^k$ we have that elements in $\cG_{(\ell,d)}^k(f^{(k)},f^{(k)})$ 
are  identified with $H_k^{\boldsymbol{\lambda}}$-orbits on the set 
\begin{displaymath}
\coprod_{g',g\in H_k^{\boldsymbol{\lambda}}f} \cG_{(\ell,d)}(g',g).
\end{displaymath}
Using $H_k^{\boldsymbol{\lambda}}$, we can move $g'$ to $f$, which means that 
elements in $\cG_{(\ell,d)}^k(f^{(k)},f^{(k)})$ can be identified with elements in
\begin{displaymath}
\coprod_{g\in H_k^{\boldsymbol{\lambda}}f} \cG_{(\ell,d)}(f,g)
\end{displaymath}
and this set has cardinality exactly given by \eqref{eq91}. The claim follows.
\end{proof}

\subsection{Connection to $G(\ell,k,d)$}\label{s4.5}

The vector space
\begin{displaymath}
A_{( \ell,k,d) }:=\bigoplus_{f^{(k)},g^{(k)}\in \cG_{(\ell,d)}^k} \mathbb{C}\cG_{(\ell,d)}^k(f^{(k)},g^{(k)})
\end{displaymath}
inherits from $\mathbb{C}\cG_{(\ell,d)}^{k}$ the structure of a finite dimensional associative 
algebra over $\mathbb{C}$ and there is a canonical equivalence of categories 
\begin{equation}\label{eqnk1}
\mathbb{C}\cG_{(\ell,d)}^{k}\text{-}\mathrm{mod}\cong A_{( \ell,k,d) }\text{-}\mathrm{mod}.
\end{equation} 

\begin{lemma}\label{lem61}
There is a unique homomorphism $\Psi:A_{( \ell,k,d) }\to A_{(\ell,d)}$ of algebras  such that
for $f\in\cG_{(\ell,d)}$ we have
\begin{displaymath}
\Psi(f^{(k)})=\sum_{g\in f^{(k)}}g.
\end{displaymath}
Moreover, this homomorphism is injective.
\end{lemma}

\begin{proof}
Uniqueness is clear as $\{f^{(k)}\mid  f\in\cG_{(\ell,d)}\}$ forms a basis of $A_{( \ell,k,d) }$.
Injectivity is clear as $\{f\mid  f\in\cG_{(\ell,d)}\}$ forms a basis of $A_{(\ell,d)}$ and thus 
$\Psi$ sends linear independent elements to linear independent elements. To check that 
$\Psi$ extends to a homomorphism, it is enough to  check that $\Psi$  is compatible 
with composition in $\cG_{(\ell,d)}^k$. This follows directly from the definitions.
\end{proof}

Our main observation in this section is the following.

\begin{theorem}\label{thm55}
The image of $\Phi^{-1}\Psi$ coincides with the subalgebra $\mathbb{C}[G( \ell,k,d) ]$ in
$\mathbb{C}[S(\ell,d)]$. In particular, the algebras $A_{( \ell,k,d) }$ and 
$\mathbb{C}[G( \ell,k,d) ]$ are isomorphic.
\end{theorem}

\begin{proof}
As the dimensions of $A_{( \ell,k,d) }$ and $\mathbb{C}[G( \ell,k,d) ]$ agree by \eqref{eq9}, we only need to check that 
each element in $\mathbb{C}[G( \ell,k,d) ]$ belongs to the image of $\Phi^{-1}\Psi$. For 
elements in $S_d$ this follows directly from the definitions.  Therefore it remains to check 
the statement for elements inside the subgroup $K\subset S(\ell,d)$ which corresponds to all diagonal 
matrices. This subgroup has cardinality $\ell^d$ and is generated by $s_0^{(j)}$, where 
$j=1,2,\dots,d$. From \eqref{eqn1}, for and $a_j\in\mathbb{Z}$ we have
\begin{displaymath}
\Phi:\prod_{j=1}^d\big(s_0^{(j)}\big)^{a_j}\mapsto\sum_f \xi_{\ell}^{\sum_j a_jf(j)}e_f.
\end{displaymath}
The effect of the action of $\theta_k$ (which maps the color $s$ to $s+\frac{\ell}{k}$) 
on the latter is
\begin{displaymath}
\sum_f \xi_{\ell}^{\sum_j a_j(f(j)+\frac{\ell}{k})}e_f. 
\end{displaymath}
Note that the diagonal matrix with entries $\xi_{\ell}^{a_j}$, where $j=1,2,\dots,d$, belongs to 
$G( \ell,k,d) $ if and only if $k$ divides $\sum_j a_j$. In the latter case, 
$\ell$ divides $\frac{\ell}{k}\sum_j a_j$,
which  implies that $\Phi(x)$ is $H_k$-invariant for every $x\in K\cap G( \ell,k,d) $.
Note that the group $K\cap G( \ell,k,d) $ has order $\frac{\ell^d}{k}$. 

Let $D$ be the subalgebra of $A_{(\ell,d)}$ generated by all $e_f$, where $f\in \cG_{(\ell,d)}$. 
Then the action of $H_k$ preserves $D$ and is free on $D$. Therefore the algebra $D^{H_k}$  
of all $H_k$-invariant elements in $D$ has dimension $\frac{l^d}{k}$. Comparing this with 
the previous paragraph, we thus get $\Phi(\mathbb{C}[K])=D^{H_k}$. The statement of the 
theorem follows.
\end{proof}

\subsection{Simple $\cG_{(\ell,d)}^k$-modules versus simple $\cG_{(\ell,d)}$-modules}\label{s4.4}

Consider the set $\Lambda(\ell,d)/H_k$ and let $\Gamma\subset \Lambda(\ell,d)$ be a cross-section
of the $H_k$-orbits. Choose any $\boldsymbol{\lambda}\in \Gamma$ and $\mathbf{p}\in\mathbf{T}_{\boldsymbol{\lambda}}$.
Then for each $z\in H_k$ we have $z\cdot \boldsymbol{\lambda}\in \Lambda(\ell,d)$ and
$z\cdot \mathbf{p}\in\mathbf{T}_{z\cdot \boldsymbol{\lambda}}$. Consider the $\cG_{(\ell,d)}$-module
$\mathrm{Q}_{\mathbf{p}}$ defined as follows:
\begin{itemize}
\item 
$
\mathrm{Q}_{\mathbf{p}}(f):=
\begin{cases}
\mathscr{S}_{z\cdot \mathbf{p}},&\text{if $\boldsymbol{\lambda}_f=z\cdot \boldsymbol{\lambda}$ for some $z\in H_k$};\\
0,&\text{if $\boldsymbol{\lambda}_f\not\in \boldsymbol{\lambda}^{(k)}$}.
\end{cases}
$
\item For any $f,g$ such that  $\boldsymbol{\lambda}_f=\boldsymbol{\lambda}_g=z\cdot \boldsymbol{\lambda}$ 
for some $z\in H_k$,  any $\pi\in\cG_{(\ell,d)}(f,g)$ and any $v\in \mathscr{S}_{z\cdot \mathbf{p}}$, we set 
\begin{displaymath}
\mathrm{L}_{\mathbf{p}}(\pi)\cdot v:=\sigma_{g,z\cdot f_{\boldsymbol{\lambda}}}\pi
\sigma_{z\cdot f_{\boldsymbol{\lambda}},f}(v).
\end{displaymath}
\item Extend this action to the whole of $\mathbb{C}\cG_{(\ell,d)}$ by linearity.
\end{itemize}
Comparing this with the definition of $\mathrm{L}_{\theta_k^{i}(\mathbf{p})}$ in Subsection~\ref{s2.4}, we get an isomorphism of $\cG_{(\ell,d)}$-modules
\begin{displaymath}
\mathrm{Q}_{\mathbf{p}}\cong\bigoplus_{i=1}^{[H_k:H_k^{\boldsymbol{\lambda}}]} 
\mathrm{L}_{\theta_k^{i}(\mathbf{p})}. 
\end{displaymath}
The group $H_k$ acts on $\mathrm{Q}_{\mathbf{p}}$ by permuting colors and twisting the action of 
$\cG_{(\ell,d)}$ accordingly. This action of $H_k$ is free and induces a free action of $H_k$ on the 
corresponding $A_{(\ell,d)}$-module (denoted with the same symbol)
\begin{displaymath}
Q_{\mathbf{p}}:=\bigoplus_{f\in \cG_{(\ell,d)}}\mathrm{Q}_{\mathbf{p}}(f). 
\end{displaymath}
Then $H_k$ acts by $A_{( \ell,k,d) }$-automorphism on $Q_{\mathbf{p}}$ since $A_{( \ell,k,d) }$ 
is exactly the set of fixed points in $A_{(\ell,d)}$ with respect to the action of $H_k$.

For $m\in\underline{k}$, consider the $\xi_\ell^m$-eigenspace for $\theta_k$:
\begin{displaymath}
Q_{\mathbf{p}}^m:=\{v\in Q_{\mathbf{p}}\mid  \theta_k\cdot v=\xi_\ell^m v\}. 
\end{displaymath}
Since the action of $H_k$ commutes with the action of 
$A_{( \ell,k,d) }$, the space $Q_{\mathbf{p}}^m$ is an $A_{( \ell,k,d) }$-submodule of $Q_{\mathbf{p}}$.

Assume that $f$ and $g$ are of the same type. Then, thanks to Lemma~\ref{lem51},
every element $\sigma^{(k)}\in\cG_{(\ell,d)}^k(f^{(k)},g^{(k)})$ can be uniquely written
in the form 
\begin{equation}\label{eql1}
\sigma^{(k)}=(\sigma^{(k)}_{(f,h(f))})^{t}\cdot\pi^{(k)}
\end{equation}
for some $t\in\{1,2,\dots,|H_k^{\boldsymbol{\lambda}}|\}$ and $\pi\in\cG_{(\ell,d)}(f,f)$.

Translating from $A_{( \ell,k,d) }$ to $\cG_{(\ell,d)}^{(k)}$ via the equivalence \eqref{eqnk1}, 
we get that the $A_{( \ell,k,d) }$-module
$Q_{\mathbf{p}}^m$ corresponds to the $\cG_{(\ell,d)}^{(k)}$-module
$\mathrm{L}_{(\mathbf{p},m)}$ defined as follows:
\begin{itemize}
\item $\mathrm{L}_{(\mathbf{p},m)}(f^{(k)}):=
\begin{cases}
\mathscr{S}_{\mathbf{p}},&\text{if $f\in\boldsymbol{\lambda}^{(k)}$};\\
0,&\text{if $f\notin\boldsymbol{\lambda}^{(k)}.$}
\end{cases}$
\item $\mathrm{L}_{(\mathbf{p},m)}(f^{(k)}):=\mathscr{S}_{\mathbf{p}}$ if $f\in\boldsymbol{\lambda}^{(k)}$.
\item For any $f,g$ of type $\boldsymbol{\lambda}$, any $\sigma\in\cG_{(\ell,d)}(f,g)$
and any $v\in\mathscr{S}_{\mathbf{p}}$ write $\sigma^{(k)}$ in the form \eqref{eql1} and set
\begin{displaymath}
\mathrm{L}_{(\mathbf{p},z)}(\sigma^{(k)})\cdot v:=
\xi_\ell^{\frac{\ell}{|H_k^{\boldsymbol{\lambda}}|}\cdot t\cdot m}\cdot \pi(v).
\end{displaymath}
\item Extend the action to the whole of $\cG_{(\ell,d)}^{(k)}$ by linearity.
\end{itemize}

For $\boldsymbol{\lambda}\in\Gamma$ denote by 
$\mathbf{T}^{(k)}_{\boldsymbol{\lambda}}$ the set of all pairs $(\mathbf{p},m)$, where
$\mathbf{p}\vdash\boldsymbol{\lambda}$ is a multi-partition and $m\in \underline{|H_k^{\boldsymbol{\lambda}}|}$.
Our main result here is the following:

\begin{theorem}\label{thm77}
The set of all $\mathrm{L}_{(\mathbf{p},m)}$, where 
\begin{displaymath}
(\mathbf{p},m)\in\bigcup_{\boldsymbol{\lambda}\in\Gamma} \mathbf{T}_{\mathbf{p}}^{(k)},
\end{displaymath}
is a cross-section of isomorphism classes of simple
$\cG_{(\ell,d)}^{(k)}$-modules.
\end{theorem}

\begin{proof}
The fact that each $\mathrm{L}_{(\mathbf{p},m)}$ is simple is clear as $\mathrm{L}_{(\mathbf{p},m)}$
takes non-zero values on a connected component of the groupoid $\cG_{(\ell,d)}^{(k)}$, and each nonzero
$\mathrm{L}_{(\mathbf{p},m)}(f)$ is a simple module already over $\cG_{(\ell,d)}(f,f)$, which is a subgroup of
$\cG_{(\ell,d)}^{(k)}(f,f)$.

From Theorem~\ref{thm55} we know that simple $\cG_{(\ell,d)}^{(k)}$-modules correspond to simple
$A_{( \ell,k,d) }$-modules and the latter can be obtained by restricting simple $A_{(\ell,d)}$-modules. The latter are classified by Proposition~\ref{prop9}\eqref{prop9.2}.
This and our construction of $\mathrm{L}_{(\mathbf{p},m)}$ above implies that the set of all
$\mathrm{L}_{(\mathbf{p},m)}$ (for all $\mathbf{p}$ and $m$) contains 
all simple $\cG_{(\ell,d)}^{(k)}$-module (up to isomorphism).

Finally, we see directly from the definition we see that $\mathrm{L}_{(\mathbf{p},m)}\cong 
\mathrm{L}_{(\mathbf{p},m+|H_k^{\boldsymbol{\lambda}}|)}$. From Lemma~\ref{lem51} it follows that 
$\mathrm{L}_{(\mathbf{p},m)}(f)$ for $\mathbf{p}\vdash\boldsymbol{\lambda}_f$ and
$m\in\underline{|H_k^{\boldsymbol{\lambda}_f}|}$ is a correct indexing set for simple
$\cG_{(\ell,d)}^{(k)}(f^{(k)},f^{(k)})$-modules (see e.g \cite[Subsection~8.2]{Se}). This completes the proof.
\end{proof}

For a general exposition of representation theory of wreath products, see \cite{CSST}, \cite{JK}, \cite{Se}.

\begin{corollary}\label{cor78}
Let $\boldsymbol{\lambda}\in\Lambda(\ell,d)$ and $\mathbf{p}\in\mathbf{T}_{\boldsymbol{\lambda}}$. Then
\begin{displaymath}
\mathrm{Res}^{A_{(\ell,d)}}_{A_{( \ell,k,d) }}\mathrm{L}_{\mathbf{p}}\cong
\bigoplus_{m\in\underline{|H_k^{\boldsymbol{\lambda}}|}}\mathrm{L}_{(\mathbf{p},m)}.
\end{displaymath}
\end{corollary}

\begin{proof}
This follows directly from the construction of $\mathrm{L}_{(\mathbf{p},m)}$.
\end{proof}

\subsection{Schur-Weyl duality for $G( \ell,k,d) $}\label{s4.6}

For $n=\ell m$ with $m>d$ take $\mathbf{p}=(m,m,\dots,m)\in\mathbb{Z}_{>0}^\ell$ and consider 
the action of $\mathbf{GL}_{\mathbf{p}}$ on $V=\mathbb{C}^{n}$ as in Subsection~\ref{s3.2}. 
Consider further the free product $\mathbf{GL}_{\mathbf{p}}*\mathbb{Z}$ and let 
the generator $1$ of $\mathbb{Z}$ act on $\mathbb{C}^{n}$ by mapping $v_i$ to 
$v_{i+\frac{\ell}{k}m}$, with the convention that $v_{s}=v_{s-n}$ for $s>n$. Our main observation 
in this subsection is the following result (for $k=\ell$ the quantum version of this
result appears in \cite{HS}).

\begin{theorem}\label{thm95}
The action of  $\mathbf{GL}_{\mathbf{p}}*\mathbb{Z}$ and the restricted action of 
$A_{( \ell,k,d) }$ on $V^{\otimes d}$ have the double centralizer property 
\begin{displaymath}
\xymatrix{
V^{\otimes d} \ar@(dr,r)[]_{A_{( \ell,k,d) }} \ar@(dl,l)[]^{\mathrm{GL}_{\mathbf{p}}*\mathbb{Z}} 
} 
\end{displaymath}
in the sense that they generate each others centralizers. 
\end{theorem}

\begin{proof}
Consider the action of $\cG_{(l,k)}$ on $V^{\otimes d}$ as described in Subsection~\ref{s3.3}.
The left action of the additional group $\mathbb{Z}$ corresponds to the permutation of colors
as given by the group $H_k$. Therefore the subalgebra $A_{( \ell,k,d) }$ of $H_k$-invariants in
$A_{(\ell,d)}$ centralizes the action of $\mathbf{GL}_{\mathbf{p}}*\mathbb{Z}$. Since the action 
of $H_k$ is free and the action of $A_{(\ell,d)}$, and hence also of $A_{( \ell,k,d) }$, is faithful
(see Subsection~\ref{s3.7}), the double centralizer property follows by locally comparing 
dimensions (as the left action increased by a factor of $k$ while the right action decreased 
by the same factor). 
\end{proof}


\noindent
Volodymyr Mazorchuk, Department of Mathematics, Uppsala University,
Box 480, 751 06, Uppsala, SWEDEN, {\tt mazor\symbol{64}math.uu.se}

\noindent
Catharina Stroppel, Mathematisches Institut, Universit{\"a}t Bonn,\\
Endenicher Allee 60, D-53115, Bonn, GERMANY,\\
{\tt stroppel\symbol{64}math.uni-bonn.de}
\end{document}